\documentclass{article}
\usepackage[margin=2.9cm]{geometry}             
\usepackage{enumerate,bm,amsmath,amsthm,amsfonts}
\usepackage{paralist}
\usepackage{cleveref}
\DeclareMathOperator{\real}{Re}
\DeclareMathOperator{\imag}{Im}
\newcommand{\C}{\mathbb C}

\newcommand{\R}{\mathbb R}

\newcommand{\dopt}[2]{\frac{\partial #1}{\partial #2}}

\newcommand{\diffable}[1]{{C}^{#1}}

\DeclareMathOperator{\wt}{wt}
\DeclareMathOperator{\Tr}{Tr}
\newtheorem{theorem}{Theorem}
\newtheorem{lemma}[theorem]{Lemma}

\theoremstyle{remark}
\newtheorem{remark}[theorem]{Remark}
\newtheorem{example}[theorem]{Example}
\theoremstyle{definition}
\newtheorem{definition}[theorem]{Definition}
\numberwithin{equation}{section}
\DeclareMathOperator{\Aut}{Aut}
\DeclareMathOperator{\trace}{Tr}

\numberwithin{equation}{section}

\begin{document}

\title{A sphericity criterion for strictly pseudoconvex hypersurfaces 
  in $\mathbb{C}^2$ via invariant curves}

\author{Florian Bertrand, Giuseppe Della Sala and Bernhard Lamel}
\maketitle 
\begin{abstract}
We prove that if every chain on a strictly pseudoconvex hypersurface $M$ in $\mathbb{C}^2$ coincides with the boundary of a stationary disc, then $M$ is locally spherical. 
\end{abstract} 

\section*{Introduction}
Every strictly pseudoconvex hypersurface $M\subset \C^2$ bounding a domain $\Omega\subset \C^2$
carries two natural, biholomorphically invariant families of real curves: the so-called {\em chains} and boundaries of {\em stationary discs}. These come from very different types of 
geometrical constructions. {Chains} have been introduced by Chern and Moser \cite{Chern:1974wu} as the CR geometry analogue of geodesics in Riemannian geometry. Stationary 
discs, on the other hand, are the solutions to the Euler-Lagrange equations of 
Kobayashi extremal discs in $\Omega$. If $M$ is in 
addition real-analytic, it carries a third natural biholomorphically
invariant family of real curves: traces of Segre varieties. A 
theorem of Faran \cite{FaranV:1981iz} shows that if the traces of Segre 
varieties agree with the chains, then $M$ is locally spherical. In a 
former paper \cite{be-de-la} we showed that if the traces of Segre varieties 
agree with the traces of stationary discs, then $M$ is also locally 
spherical. 

In this paper, we address the remaining question: If the traces of 
stationary discs coincide with chains, is $M$ also necessarily spherical? We
have been asked this repeatedly when presenting the results in \cite{be-de-la}, 
and it turns out that the answer is also yes. The natural setting
for this question is for sufficiently smooth hypersurfaces. 

\begin{theorem}
\label{thm:main} Assume that $M$ is  a strictly pseudoconvex hypersurface of class  $
\diffable{12}$
in $\mathbb{C}^2$. If the chains of $M$ are boundaries of stationary 
discs, then $M$ is locally spherical. 
\end{theorem}

We remark that $M$ in Theorem~\ref{thm:main} is not assumed to be closed (so the theorem is a local result).
In 
order to prove this theorem, we cannot utilize the cited results. 
Instead, we rely on Fefferman's characterization of chains as projections of 
light rays of an associated Lorentz metric and analyzing its 
Hamiltonian. We construct a special family of chains centered at the origin and
show that if each of the members of this family is the trace of
a stationary disc, then the origin is an umbilical point. 

The organization of this paper is to review the basics in \cref{sec:prelim}, summarize facts about 
the chains in \cref{sec:Ham}, and give the proof of a (slightly sharpened) version of the 
theorem in \cref{sec:proof}.

\section{Preliminaries}\label{sec:prelim}

\subsection{Intrinsic geometry of strictly pseudoconvex hypersurfaces} 

In this section we give a quick review of the basic local biholomorphic 
equivalence theory for strictly pseudoconvex hypersurfaces in $\C^2$. We thank one of the referees for the suggestion to include some background material in this paper, and hope the reader will enjoy it.

A hypersurface $M\subset \C^2$ inherits
the complex structure of $\C^2$ on its complex tangent spaces $T^c_p M = T_p M \cap i T_p M$. The complex tangent
bundle $T^c M$ is defined by the vanishing of a not uniquely determined $1$-form $\theta$, which we 
usually call {\em characteristic form}; the annihilator $(T^c M)^\perp = T^0 M = \mathcal{N} M \subset T^* M$ is called the 
{\em characteristic} or {\em conormal} bundle of $M$. If  $d \theta$ induces a hermitian 
inner product on $T^c M$ by $h(X,Y) = i \theta ( [X,\bar Y]) = - i d\theta (X,\bar Y)$, called the Levi form of $M$ with 
respect to $\theta$, then we say  
that $M$ is {\em strictly pseudoconvex}. 

Thus, geometrically speaking, a strictly pseudoconvex hypersurface $M\subset \C^2$ can be thought of as a 
$3$-dimensional manifold with a contact structure with some additional compatibility conditions. The choice of 
a contact form gives rise to a {\em pseudohermitian structure} $(M,\theta)$. 

The ambivalence in choosing a contact form gives rise to fascinating mathematics which mixes
aspects of complex, conformal, contact, and symplectic geometry. The {\em equivalence problem} for 
strictly pseudoconvex hypersurfaces in $\C^2$ was solved in a series of groundbreaking papers by 
E. Cartan \cite{Cartan:1932ws,Cartan:1933ux} applying his method of equivalence; later, Tanaka \cite{Tanaka:1962ti} and Chern-Moser \cite{Chern:1974wu} gave 
solutions of the problem for strictly pseudoconvex and Levi-nondegenerate hypersurfaces in higher dimensions. For the convenience of the 
reader, we recall some of the necessary background, with a view towards the Fefferman construction 
of chains which we are going to use.

If we start with an arbitrary characteristic 
form $\theta$ we can consider a real line bundle $E$ over $M$ consisting of the multiples $u \theta$, where $u>0$. 
The form $\omega = u \theta$ is intrinsically defined, and (on $E$) we have 
 \[ d\omega = i g_{1 \bar 1} \omega^1 \wedge \bar \omega^1 + \omega\wedge \varphi \]
 with a real one form $\varphi$; the forms $\omega$, $\omega^1$, $\bar \omega^1$,  and $\varphi$ span 
 $\C T^* E$. Since we assume that $M$ is strictly pseudoconvex, we have $g_{1 \bar 1} \neq 0$, and we will assume for simplicity that $g_{1 \bar 1} = 1$; the more general case where $g_{1 \bar 1}$ is not assumed to be constant follows in a similar but more involved way. Since this short review is only meant to recall how the main computations work, we opted to keep the simple variant. We follow the 
 notation of Chern-Moser \cite{Chern:1974wu} so that the reader can pick up the 
 necessary modifications in that source easily. 

 Now any other frame of $\C T^* E$ satisfying the condition above is given by 
 \[
 \begin{pmatrix} \tilde \omega \\
 \tilde \omega^1 \\
 \bar {\tilde \omega}^1 \\
 \tilde \varphi 
 \end{pmatrix} = 
 \begin{pmatrix} 1 & 0 & 0 & 0 \\
 \lambda & \mu & 0 & 0 \\
 \bar \lambda & 0 & \bar \mu & 0 \\
 s & i  \mu \bar \lambda & -i  \bar \mu \lambda & 1
 \end{pmatrix}
 \begin{pmatrix} \omega \\
  \omega^1 \\
  \bar \omega^1 \\
  \varphi 
  \end{pmatrix}
   \]
   where $|\mu|^2 =1$. The group of matrices of this form is denoted by $G_1$, and 
   we can form the principal $G_1$-bundle $Y$ over $E$; thus $s$, $\lambda$, and $\mu$ are fiber coordinates. One 
   has the integrability condition
   \[ d\omega^1 = \omega^1 \wedge \alpha + \omega \wedge \beta, \]
   for every frame as above, with some not uniquely determined forms $\alpha$ and $\beta$. 
   
   One then obtains a uniquely determined frame $\omega, \omega^1, \bar \omega^1, \varphi, \alpha, \beta, \psi$ of 
   $\C T^* Y$ satisfying a number of identities: 

   \begin{theorem}[\cite{Cartan:1932ws}] \label{thm:intrinsic} There exists a unique frame 
   $\omega, \omega^1, \bar \omega^1, \varphi, \alpha, \beta, \psi$ of 
   $\C T^* Y$ and invariantly defined functions $Q,R$ such that 
   \[ \begin{aligned}
   d\omega &= i  \omega^1 \wedge \bar \omega^1 + \omega\wedge \varphi \\ 
   d\omega^1 &= \omega^1 \wedge \alpha + \omega \wedge \beta \\
   \varphi &= \alpha + \bar \alpha \\
   d\varphi &=    i \omega^1 \wedge \bar \beta 
   -i  \bar \omega^1 \wedge \beta +  \omega \wedge \psi \\
   d \alpha &= i \bar \omega^1 \wedge \beta + 2i \omega^1 \wedge \bar \beta- \frac{\psi}{2}\wedge \omega \\
   d \beta &= \bar \alpha \wedge \beta - \frac{\psi}{2} \wedge \omega^1 + Q \bar \omega^1 \wedge \omega \\ 
   d \psi &= \varphi \wedge \psi + 2 i \beta \wedge \bar \beta + (R \omega^1 + \bar R \bar \omega^1) \wedge \omega
   \end{aligned} \]
   \end{theorem}

\begin{definition}
A curve $\gamma$ is called a {\em chain} if it solves
the system of ODEs $\omega^1 = \beta =0$.
\end{definition}

We note that the equations for the real 
forms $\omega$, $\varphi$, and $\psi$ along a chain simply read 
\[ d\omega = \omega \wedge \varphi , \quad d\varphi = \omega\wedge\psi, \quad d\psi = \varphi \wedge \psi. \]
One can use these equations to therefore introduce a 
canonical parameter along the chain defined up to a linear fractional map. 

We now give, as promised, the details for the construction of the canonical forms in Theorem~\ref{thm:intrinsic}, basically to set the stage. We also point the reader to the 
 book by Jacobowitz \cite{MR1067341}. 

    First, one observes that 
   exterior differentiation of the frame conditions 
   \begin{equation}
      \label{e:frame} \begin{aligned} d\omega &= i  \omega^1 \wedge \bar \omega^1 + \omega\wedge \varphi \\
    d\omega^1 &= \omega^1 \wedge \alpha + \omega \wedge \beta \end{aligned}
   \end{equation}
    yields 
   \begin{equation}\label{e:dframe}
   \begin{aligned} 
   0 &=  i d \omega^1 \wedge \bar \omega^1 - i  \omega^1 \wedge d \bar \omega^1 + d \omega \wedge \varphi - \omega \wedge d \varphi  \\
   & = i \left( \omega^1 \wedge \alpha + \omega \wedge \beta \right) \wedge \bar \omega^1 - i  \omega^1 \wedge \left( \bar \omega^1 \wedge \bar \alpha +  \omega \wedge \bar \beta \right) \\&\qquad + \left( i  \omega^1 \wedge \bar \omega^1 + \omega\wedge \varphi \right) \wedge \varphi - \omega \wedge d \varphi  \\ 
   &=  i (- \alpha -\bar \alpha + \varphi) \wedge  \omega^1 \wedge \bar \omega^1 + \left( - d\varphi + i \beta \wedge  \bar \omega^1 -  i \bar \beta \wedge \omega^1    \right)\wedge \omega  \\
   0 &= d\omega^1 \wedge \alpha - \omega^1 \wedge d\alpha + d\omega \wedge \beta - \omega \wedge d \beta \\
   &=  \omega \wedge \beta \wedge \alpha - \omega^1 \wedge d \alpha + ( i \omega^1 \wedge \bar \omega^1 + \omega\wedge \varphi ) \wedge \beta - \omega \wedge d\beta\\
   &= (-d\alpha + i \bar \omega^1 \wedge \beta)\wedge \omega^1 +  (\beta\wedge \alpha + \varphi\wedge \beta - d \beta ) \wedge \omega ,
   \end{aligned}
   \end{equation}
   It follows from the first equation in \eqref{e:dframe} that $- \alpha -\bar \alpha + \varphi = A \omega^1 + \bar A \bar \omega^1 + C \omega$ with $C=\bar C$, 
   i.e. with the choice $\tilde \alpha = \alpha + A \omega^1 + \frac{C}{2} \omega$ we have $\varphi = \tilde \alpha + \bar {\tilde \alpha}$. It is easy to see that $\varphi$
   with this property are unique up to multiples of $\omega$, and for such a choice of $\varphi$, we have 
   $  - d\varphi  + i \beta \wedge  \bar \omega^1 -  i \bar \beta \wedge \omega^1 = - \omega \wedge \psi $ for a real $1$-form $\psi$.  
   Summarizing, we have imposed the following restrictions: 
   \[ 
   \begin{aligned} 
   d\omega &= i  \omega^1 \wedge \bar \omega^1 + \omega\wedge \varphi \\
   d\omega^1 &= \omega^1 \wedge \alpha + \omega \wedge \beta \\ 
   \varphi &= \alpha + \bar \alpha \\
   d\varphi &=    i \omega^1 \wedge \bar \beta 
   -i  \bar \omega^1 \wedge \beta +  \omega \wedge \psi
   \end{aligned}
   \]
   and after simplification of the second equation in \eqref{e:dframe}, we  have the equation
   \begin{equation}\label{e:dframe2}
      (d\alpha - i \bar \omega^1 \wedge \beta)\wedge \omega^1 +  ( d \beta - \bar \alpha \wedge \beta ) \wedge \omega = 0.
   \end{equation}

   The forms $\alpha$, $\beta$, and $\psi$ satisfying these 
   identities are uniquely determined up to a change 
   \begin{equation}\label{e:framechange}
       \begin{aligned} 
   \tilde \alpha &= \alpha + D \omega \\ 
   \tilde \beta & = \beta + D \omega^1 + E \omega \\
   \tilde \psi &= \psi + G \omega + i ( \bar E \omega^1-  E \bar \omega^1)  
   \end{aligned}
   \end{equation}
   where $D$ is purely imaginary and $G$ is real. 

One can then check that the 
form $\Phi = d\alpha - i \bar \omega^1  \wedge \beta - 2i  \omega^1 \wedge \bar  \beta $
satisfies $\Phi = - \bar \Phi$ modulo $\omega$:  
\[ \begin{aligned}
\Phi + \bar \Phi  &= d\alpha - i \bar \omega^1  \wedge \beta - 2i  \omega^1 \wedge \bar  \beta + d \bar \alpha + i  \omega^1  \wedge \bar  \beta + 2i \bar \omega^1 \wedge   \beta\\ &= 
 d \varphi  -  i  \omega^1  \wedge \bar  \beta +i \bar \omega^1  \wedge \beta  \\ 
 &= \omega \wedge \psi 
\end{aligned} \]
Since in addition $\Phi\wedge \omega^1 = 0$ modulo 
$\omega$ by \eqref{e:dframe2}, we have 
\[ d\alpha - i\bar \omega^1  \wedge \beta - 2i  \omega^1 \wedge \bar  \beta \cong S \omega^1 \wedge \bar \omega^1  \mod \omega \]
 and the real number $S$ transforms in the following way under a change of frame as above (computing modulo $\omega$):
 \[ \begin{aligned}
 \tilde S  \omega^1 \wedge \bar \omega^1 &= d \tilde \alpha - i \bar{{\omega}}^1 \wedge \tilde \beta  - 
 2i  \omega^1 \wedge \overline{\tilde  \beta} \\ 
 &= d (\alpha + D \omega) - i \bar{{\omega}}^1 \wedge (\beta + D \omega^1  ) - 2i \omega^1 \wedge (\bar \beta - D \bar \omega^1  ) \\ 
 & = d \alpha + D \wedge d\omega - i \bar \omega^1 \wedge \beta +  3 i D \omega^1 \wedge \bar \omega^1 -2i \omega^1 \wedge \bar \beta \\ 
 & = (S + 4 i D) \omega^1 \wedge \bar \omega^1. 
 \end{aligned}
   \] 
The condition $S = 0$ therefore is possible for a certain $D$, and fixes the form $\alpha$ uniquely. We 
proceed to calculate with  
\[ \Phi = d\alpha - i \bar \omega^1 \wedge \beta -2i  \omega^1 \wedge \bar  \beta = \lambda \wedge \omega
\] for some $1$-form $\lambda$, and 
we note right away that 
\[ \begin{aligned}(\lambda + \bar \lambda) \wedge \omega &= 
\Phi + \bar \Phi \\ 
&=  
\omega \wedge \psi, \end{aligned}  \]
in other words, $\lambda + \bar \lambda = - \psi$ 
modulo $\omega$. Plugging  $d\alpha - i \bar \omega^1 \wedge \beta -  2i  \omega^1 \wedge \bar  \beta = \lambda \wedge \omega$ into \eqref{e:dframe2}, we obtain  
\[ \begin{aligned} 0 &= (d\alpha - i \bar \omega^1 \wedge \beta)\wedge \omega^1 +  ( d \beta - \bar \alpha \wedge \beta ) \wedge \omega \\ 
&= ( d \beta - \bar \alpha \wedge \beta +\omega^1 \wedge \lambda  )\wedge \omega,
\end{aligned} \]
so that 
\[  d\beta - \bar \alpha \wedge \beta + \omega^1 \wedge \lambda = \mu \wedge \omega \]
for some other $1$-form $\mu$. 

We next take the derivative of  $d\alpha - i \bar \omega^1 \wedge \beta - 2i  \omega^1 \wedge \bar  \beta= \lambda \wedge \omega$, which modulo 
$\omega$ yields 
\[ \begin{aligned} 0&= 
 i d \bar \omega^1 \wedge \beta - i \bar \omega^1 \wedge d \beta  +2i d\omega^1 \wedge \bar \beta - 2i \omega^1 \wedge d\bar \beta - \lambda \wedge d \omega \\ 
 &= i \bar \omega^1 \wedge (\bar \alpha \wedge \beta - d \beta) + 2i \omega^1 \wedge (\alpha \wedge \bar \beta - d \bar \beta) 
  - i \lambda \wedge \omega^1 \wedge \bar \omega^1 \\ 
 &= -i  \bar \omega^1 \wedge ( \omega^1 \wedge \lambda) - 2i \omega^1 \wedge (\bar \omega^1 \wedge \bar \lambda)  
  - i \lambda \wedge \omega^1 \wedge \bar \omega^1  \\ 
 &= 2i\bar \omega^1 \wedge \omega^1 \wedge \bar \lambda.
\end{aligned} \]

It follows that we can write $\lambda = - \frac{\psi}{2} +  V \omega^1 -  \bar V \bar \omega^1+  a \omega$,  and we have a complete expression 
for $\lambda$, and therefore, for
\[ \Phi = d\alpha - i \bar \omega^1 \wedge \beta - 2i \omega^1 \wedge \bar \beta = 
\lambda\wedge \omega = -\frac{\psi}{2}\wedge \omega +  V \omega^1 \wedge \omega -  \bar V \bar \omega^1 \wedge \omega , \]
and the effect of the frame change
\[ \begin{aligned} 
   \tilde \beta & = \beta +  E \omega \\
   \tilde \psi &= \psi + G \omega + i ( \bar E \omega^1-  E \bar \omega^1)  
   \end{aligned}
   \]
yields $\tilde V = V - \frac{3i}{2} \bar E $. We can thus choose $\beta$ uniquely requiring that 
$V=0$
and now also have 
\begin{equation}\label{e:dalpha}
    d \alpha = i \bar \omega^1 \wedge \beta + 2i \omega^1 \wedge \bar \beta- \frac{\psi}{2}\wedge \omega
 \end{equation} 

Substituting (\ref{e:dalpha}) back into 
\eqref{e:dframe2} yields 
\[ (d\alpha - i \bar \omega^1 \wedge \beta)\wedge \omega^1 +  ( d \beta - \bar \alpha \wedge \beta ) \wedge \omega = \left( d \beta - \bar \alpha \wedge \beta + \frac{\psi}{2}\wedge \omega^1 \right) \wedge \omega , \]
so that 
\begin{equation}\label{e:defnu}
   d \beta - \bar \alpha \wedge \beta + \frac{\psi}{2}\wedge \omega^1= \nu\wedge \omega.
\end{equation}

Taking the derivative of $ d\varphi = i \omega^1 \wedge \bar \beta  -i  \bar \omega^1 \wedge \beta +  \omega \wedge \psi
  $ and using \eqref{e:frame} and \eqref{e:defnu}, we obtain 
\[ \left( d\psi - \varphi\wedge \psi - 2i \beta \wedge \bar \beta + i \omega^1 \wedge \bar \nu - i \bar \omega^1 \wedge \nu \right) \wedge \omega = 0.
\]
We can therefore write 
\begin{equation}\label{e:dpsi}
    \Psi = d\psi - \varphi\wedge\psi - 2i \beta\wedge \bar \beta = i \bar \omega^1\wedge \nu 
-i \omega^1 \wedge \bar \nu + \varrho \wedge \omega. 
\end{equation}

In the next (and last) step, we take the exterior derivative of \eqref{e:defnu}, obtaining after 
using \eqref{e:dalpha}, \eqref{e:defnu}, \eqref{e:dpsi}, and \eqref{e:frame}, computing modulo $\omega$:
\[
\begin{aligned}
0&= - d\bar \alpha \wedge \beta + \bar \alpha \wedge d \beta + \frac{d\psi}{2} \wedge \omega^1 
- \frac{\psi}{2}\wedge d \omega^1  - d \nu\wedge \omega + \nu\wedge d \omega \\ 
&= - i  \omega^1 \wedge \bar \beta \wedge \beta - \bar \alpha \wedge  \frac{\psi}{2}\wedge \omega^1 +\frac{1}{2} \left( \varphi\wedge\psi + 2i \beta\wedge \bar \beta + i \bar \omega^1\wedge \nu 
 \right) \wedge \omega^1 \\ 
 &\quad\quad - \frac{\psi}{2} \wedge \omega^1\wedge \alpha + \nu \wedge (i \omega^1 \wedge \bar \omega^1) \\ 
 & = -\frac{3i}{2} \nu \wedge\omega^1 \wedge \bar \omega^1 + \frac{1}{2} \underbrace{\left( \varphi - \alpha - \bar \alpha  \right)}_{=0} \wedge \psi\wedge \omega^1. 
\end{aligned}
 \]

Since $\nu$ is only defined modulo $\omega$, 
we can therefore write $ \nu = P \omega^1 + Q \bar \omega^1 $. It turns out that the exterior differentiation 
of \eqref{e:dalpha}, using the expressions for $d\beta$ and $d\psi$ already obtained, 
implies that $P=\bar P$ is real. Hence from \eqref{e:dpsi} we have

 \[ \Psi =  i  \bar \omega^1\wedge \nu 
-i \omega^1 \wedge \bar \nu + \varrho \wedge \omega = 2 i P  \bar \omega^1\wedge \omega^1 + 
  \varrho \wedge \omega.  \]
 The last free parameter $G$ in the frame change $\tilde \psi = \psi + G \omega$ transforms 
 $ \tilde P =  P + G $, and we finally have an invariantly defined frame as we wanted by 
 requiring $P=0$. Note that with this choice we have 
\begin{equation}
   \label{e:dbetafin}{} d \beta = \bar \alpha \wedge \beta - \frac{\psi}{2} \wedge \omega^1 + Q \bar \omega^1 \wedge \omega.
\end{equation}
We can also differentiate the equation for $\Psi$ above, and obtain from it that $\rho \wedge \omega^1 \wedge \bar \omega^1 = 0$, 
so that $\rho$ is a linear combination $R \omega^1 + S \bar \omega^1$, which by reality of $\Psi$ also implies $S = \bar R$. Hence we can write 
\begin{equation}
   \label{e:dpsifin} d \psi = \varphi \wedge \psi + 2 i \beta \wedge \bar \beta + (R \omega^1 + \bar R \bar \omega^1) \wedge \omega.
\end{equation}

\subsection{The canonical connection} 
\label{sub:the_canonical_connection}
Theorem~\ref{thm:intrinsic} has a neat description in terms of a canonical connection for $Y$, 
which we recall again using the notation of \cite{Chern:1974wu}. One sets 
\[ h = \begin{pmatrix}
   0 & 0 & -\frac{i}{2} \\[4pt] 
   0 & 1 & 0 \\[4pt] 
   \frac{i}{2} & 0 & 0 
\end{pmatrix}, \qquad \pi = \begin{pmatrix}
   - \frac{2\alpha + \bar \alpha}{3} &\omega^1  & 2 \omega \\[4pt] 
    - i \bar \beta& \frac{\alpha-\bar \alpha}{3} & 2 i \bar \omega^1 \\[4pt] 
   -\frac{\psi}{4} & \frac{\beta}{2} & \frac{2 \bar \alpha + \alpha}{3}
\end{pmatrix}, \]
so that $ \pi h  + h \pi^*  = 0$ and $\trace \pi = 0$, in other words, $\pi$ is $\mathfrak{su} (2,1)$-valued (with 
the hermitian form given by $h$). It turns out that the equations of Theorem~\ref{thm:intrinsic} are 
equivalent to
\[ d\pi - \pi \wedge \pi = \begin{pmatrix}
   0 & 0 & 0 \\[4pt] 
   -i Q \bar \omega^1 \wedge  \omega & 0 & 0 \\[4pt]
   -\frac{1}{4}(R \omega^1 + \bar R \bar \omega^1) \wedge \omega & \frac{1}{2} \bar Q  \omega^1 \wedge \omega & 0
\end{pmatrix} \]


\subsection{The Chern-Moser normal form and chains} 
It is well known that the group of germs of biholomorphisms $G = \Aut (\mathbb{H}^2,0)$ of the 
Heisenberg hypersurface $\mathbb{H}^2 \subset \C^2_{{z_2},{z_1}}$ (defined by $\real {z_1} = |{z_2}|^2$)
fixing the origin are explicitly given by 
\begin{equation}\label{e:autos}
\begin{aligned} H({z_1},{z_2})&= (g(z_1,z_2),f(z_1, z_2 )) \\ &= \left(  
 \frac{|\lambda|^2{z_1}}{{1 + 2  \bar a {z_2} +  ( |a|^2 +  i t) {z_1}}} ,
 \frac{ \lambda({z_2} + a{z_1})}{1 + 2  \bar a {z_2} +  ( |a|^2 +  i t) {z_1}}  \right).  \end{aligned}\end{equation}
They are therefore uniquely determined by the derivatives $(f_{z_2}(0),f_{z_1}(0),\imag g_{z_1^2} (0)) \in \C^* \times \C \times \R$. 
Using this, we identify $G$ with $\C^* \times\C\times\R$. 

Next, we recall the celebrated Chern-Moser theorem \cite{Chern:1974wu}. 
If we   consider a germ of a strictly 
pseudoconvex real-analytic hypersurface $(M,p) \subset (\C^2,p)$, then after 
an affine change of coordinates,  $p=0$ and $M$ is given 
near $p$  by an equation of the form
\[ \real {z_1} = |{z_2}|^2 + \varphi ({z_2}, \bar {z_2}, \imag {z_1}) = |{z_2}|^2 + \sum_{\alpha,\bar \beta} \varphi_{\alpha,\bar \beta} (\imag z_1) {z_2}^\alpha \bar {z_2}^{\bar \beta}.  \]
The Chern-Moser normal form imposes conditions on the $\varphi_{\alpha, \bar \beta}$ which make this coordinate choice unique up to 
a parameter $\Lambda \in G$  in the isotropy group of $\mathbb{H}_2$:  

\begin{theorem}[Chern-Moser \cite{Chern:1974wu}, n=2]  Let $(M,p)$ be a real-analytic hypersurface. Then 
there exists a holomorphic 
choice of holomorphic coordinates $(z,w)$ 
in which $p=0$ and the equation of $M$ satisfies the {\em normalization conditions}
\begin{equation*}
  \label{e:normal} \varphi_{\alpha, \bar \beta} (\imag z_1 ) = 0, \text{ if } 
   \min(\alpha,\bar \beta) \leq 1 \text{ or } (\alpha, \beta) \in \left\{ (1,1), (2,2), (3,3), (2,3),(3,2) \right\}.
\end{equation*}
Any other choice $(\tilde z, \tilde w)$ of holomorphic coordinates in which the defining equation of $M$ takes this form 
is given by $(\tilde z, \tilde w) = H_\Lambda (z,w)$ for some $\Lambda \in G$,  with  $H_\Lambda$ uniquely determined by  
the requirement that it agrees with the map in \eqref{e:autos} up to order two. 
\end{theorem}

The lowest order term in the defining equation $\varphi$ of $M$ in normal coordinates
which is not necessarily vanishing is therefore of the form 
 $A_p z_2^2 \bar z_2^4 + \bar A_p z_2^4 \bar z_2^2$. The number $A$ transforms nicely under 
 changes of normal coordinates and is called {\em Cartan's cubic tensor}. 
 It already appears in Cartan's early work \cite{Cartan:1932ws,Cartan:1933ux}, 
 and it being $0$ is a biholomorphic invariant; points where $A_p = 0$ are called {\em umbilical points}. 
 For ease of notation later on, we always normalize $A_p$ to
 be real. Vanishing of $A_p$ on an open subset is equivalent to local sphericity: 
 
\begin{theorem}[Cartan's umbilical tensor \cite{Cartan:1932ws,Cartan:1933ux}]  Let $(M,p)\subset \C^2$ be 
a germ of a smooth strictly pseudoconvex hypersurface. If $A_q =0 $ for $q$ in a neighbourhood 
of $p$ in $M$, then $M$ is near $p$ CR-equivalent to $\mathbb{H}^2$. 
\end{theorem}

The theorem actually hods true for a 
  $\diffable{6}$ hypersurface, for which we can still define the cubic tensor, and can even (due to recent not yet published work of Kossovskiy) 
  be formulated in lower regularity. 

$A_p$ can be thought of as a form of intrinsic curvature, and in a similar vein, 
Chern and Moser used their normal form to introduce the notion of 
{\em chains}
as replacements for geodesics in Riemannian geometry. For each 
$\Lambda = (\lambda, a, t) \in G$, we obtain a
 parametrized curve (defined for $|s|$ small 
enough)
\[ \gamma(s) = H_\Lambda (is,0).  \]
If one disregards the parametrization of $\gamma$, then it turns out that 
the condition $\varphi_{2,\bar 3} = \varphi_{3, \bar 2}= 0$ is a {\em second order ODE}
whose solution is unique given $a$ (which one thinks of as a vector transverse to the complex tangent space $T^c_0 M$). 
The rest of the data in $\Lambda$ geometrically corresponds to a choice of frame of $T^c_0 M$ and a choice of parametrization of  
$\gamma$
amongst a family of projectively equivalent ones.  The second order differential 
equations for chains are not easy to compute from a defining equation of $M$. For 
boundaries of strictly pseudoconvex domains, the best way to get a computational
handle on chains for our problem turned out to be their interpretation as projections of light rays 
of an associated Lorentz metric introduced by Fefferman \cite{Fefferman:1976wl} which 
we discuss in the next section. 

\subsection{Chains and the Fefferman Hamiltonian}

We will now recall the Fefferman metric for a  strictly pseudoconvex hypersurface 
$M=\{ \rho= 0\}\subset \C^2$. We write $z_j=x_j+iy_j$, $j=1,2$, and we assume that $(y_1,x_2,y_2)$ are local coordinates on $M$ near the origin, 
which we assume to be defined by
\[ \rho (x_1,y_1,x_2, y_2) = x_1 - (x_2^2 + y_2^2) - \varphi (y_1,x_2,y_2),  \]
where $\varphi$ vanishes to order at least $3$.

The constructive appeal of Fefferman's metric is based on the fact that for 
the complex Monge-Amp\`ere operator
$$ J(\rho)= \det \left(\begin{matrix}
\rho & \rho_{\overline z_1}&   \rho_{\overline z_2}     \\
 \rho_{ z_1}& \rho_{z_1 \overline z_1} & \rho_{z_1 \overline z_2} \\
\rho_{ z_2} &  \rho_{z_2 \overline z_1} &   \rho_{z_2 \overline z_2}  \\
\end{matrix}\right),$$
one can construct approximate solutions $\rho^{(k)}$ to the equation 
$J(\rho^{(k)}) = 1 + O(\rho^{k+1})$ in an iterative way, in this particular case by 
\[ \rho^{(1)} = \frac{\rho}{\sqrt[3]{J(\rho)}}, \quad \rho^{(2)} = \rho^{(1)}
\left(\frac{5-J(\rho^{(1)})}{4} \right).  \]

The Fefferman metric is defined on a circle bundle over $M$. 
We denote by $(x_0,y_1,x_2,y_2)$ the coordinates on $\mathbb{S}^1\times M$. 
 The conjugate momenta will be denoted by $p_{x_0}, p_{y_1}, p_{x_2}$ and $p_{y_2}$. 
There is a lot of flexibility in which metric is actually used, because 
the light rays of conformally equivalent Lorentz metrics are the same. 
The one defined in \cite{Fefferman:1976wl} is 
\[ ds^2 = - \frac{i}{3} \left( \partial \rho^{(2)} - \bar\partial \rho^{(2)} \right)d x_0
+ \sum_{j,k=1}^2 \dopt{^2\rho^{{(2)}}}{z_j \bar z_k} dz_j d \bar z_k.   \]

Setting 
\begin{equation*}
\begin{aligned}
\Phi &= J(\rho), \\ 
A^{}&=  \begin{pmatrix}
0 & i\rho_{\overline z_1}&  i \rho_{\overline z_2}     \\
-i \rho_{ z_1}& 3 \rho_{z_1 \overline z_1} & 3 \rho_{z_1 \overline z_2} \\
-i \rho_{ z_2} &3 \rho_{z_2 \overline z_1} & 3 \rho_{z_2 \overline z_2}  \\
\end{pmatrix}, \\ P&=(p_{x_0},ip_{y_1},p_{x_2}+ip_{y_2}) \\
\overline{\partial}\Phi&=\left(0,\Phi_{\overline z_1},\Phi_{\overline{z}_2}\right) \\ 
\tilde{\Phi} &=\left(3\Phi_{j\overline{k}}-\frac{5}{\Phi}\Phi_j\Phi_{\overline{k}}\right)_{j,k},
\end{aligned}
\end{equation*}
the 
Hamiltonian of Fefferman's metric is now given by
\begin{equation}\label{eqH}
H= P A^{-1} P^*
-\frac{2p_{x_0}}{\Phi}
\imag \left( \overline{\partial}\Phi \cdot A^{-1}\cdot P^*
\right)
-\frac{p_{x_0}^2}{2\Phi} \Tr \left(\tilde{\Phi} A^{-1}\right),
\end{equation}
where $\Tr (\tilde{\Phi} A^{-1})$ stands for the trace of the matrix $\tilde{\Phi} A^{-1}$. {Note that the formula of the Hamiltonian in \cite[p. 410]{Fefferman:1976wl} contains a minor sign mistake, see \cite{MR407321}.}
Writing $x = (x_0, y_1,x_2,y_2)$, and $p = (p_{x_0}, p_{y_1}, p_{x_2}, p_{y_2})$,  
chains are the projections on $M$ of the solutions of the Hamiltonian system
\begin{equation}\label{eqHsys}
H(x,p) = 0, \quad x' = H_p (x,p), \quad p' = - H_x (x,p).
\end{equation}
We are now ready to discuss a basic example.
\begin{example}\label{exsphere}
In the case of the sphere $2\real z_1=|z_2|^2$, 
$$A^{-1}=
 \left(\begin{matrix}
0 & i&  0     \\
 -i & -\frac{|z_2|^2}{3}  & -\frac{z_2}{3}  \\
0 &  -\frac{\overline{z_2}}{3}  &  -\frac{1}{3}  \\
\end{matrix}\right).$$
For convenience, and since light rays for $H$ or $3H$ are the same, we consider 
   $$A^{-1}=
 \left(\begin{matrix}
0 & 3i&  0     \\
-3i & -|z_2|^2  & -z_2  \\
0 &  -\overline{z_2}  &  -1  \\
\end{matrix}\right).$$
In that case the Fefferman Hamiltonian is given by 
\begin{equation}\label{eqsphere}
H= P A^{-1} P^*
=6p_{x_0}p_{y_1}-|z_2|^2p^2_{y_1}+2y_2p_{y_1}p_{x_2}-2x_2p_{y_1}p_{y_2}-p^2_{x_2}-p^2_{y_2}.
\end{equation}
Now, we seek solution curves $\left(x_0(t),y_1(t),x_2(t),y_2(t), p_{x_0}(t), p_{y_1}(t), p_{x_2}(t),p_{y_2}(t)\right)$ to the Hamiltonian system \eqref{eqHsys}, which written out is given by
\begin{equation*}
\left\{\begin{aligned}
0&=6p_{x_0}p_{y_1}-|z_2|^2p^2_{y_1}+2y_2p_{y_1}p_{x_2}-2x_2p_{y_1}p_{y_2}-p^2_{x_2}-p^2_{y_2}\\
x_0'&= 6p_{y_1}\\
y_1'&=6p_{x_0}-2|z_2|^2p_{y_1}+2y_2p_{x_2}-2x_2p_{y_2}\\
x_2'&=2y_2p_{y_1}-2p_{x_2}\\
y_2'&=-2x_2p_{y_1}-2p_{y_2}\\
p_{x_0}'&=0\\
p_{y_1}'&=0\\
p_{x_2}'&=2x_2p^2_{y_1}+2p_{y_1}p_{y_2}\\
p_{y_2}'&=2y_2p^2_{y_1}-2p_{y_1}p_{x_2}\\
\end{aligned}\right.
\end{equation*}
To solve this system, note that we have  $p_{y_1}={\frac{c}{4}}$ for some $c$  so that we get
$x_2''=cy_2'$ and $y_2''=-cx_2'$, and therefore $z_2 = c_1 e^{-i c t} + c_2$ for some $c_1, c_2 \in \C$. It 
remains to determine $y_1$. Next we note that the quantity $y_2 p_{x_2} - x_2 p_{y_2}$ is conserved and solving one sees that the chains are the curves of the form
$$
\begin{aligned}
z_1(t) & =  \frac{1}{2}|c_1e^{-ict}+c_2|^2+i\left(\tilde{c_1}t +\tilde{c_2}\cos(ct)+\tilde{c_3}\sin(ct)+\tilde{c_4}\right)\\
z_2(t) & =  c_1e^{-ict}+c_2,
\end{aligned}
$$
where $\tilde{c_j} \in \R$ and $c_j \in \C$.   
\end{example}

We will now assign weights to all variables in the following way. The usual anisotropic scaling on $\C^2$, $\Lambda_\delta: (z_1,z_2) \mapsto (\delta z_1,\delta^2 z_2)$, $\delta>0$,  lifts to the cotangent bundle as 
$$\tilde{\Lambda}_\delta: (z_1,z_2, p_{z_1},p_{z_2}) \mapsto (\delta z_1,\delta^2 z_2,\delta^{-1}p_{z_1},\delta^{-2}p_{z_2}).$$
This leads to assign the respective natural weights $1,2,-1,-2$ and $-2$  to the variables $z_1,z_2, p_{y_1},p_{x_2}$ and $p_{y_2}$. The variables $x_0$ and $p_{x_0}$ both carry a weight $0$. However, with this convention, the Hamiltonian for the sphere \eqref{eqsphere} is homogeneous of degree $-2$. It will be more convenient for us if the Hamiltonian \eqref{eqsphere} is homogeneous of degree $2$, and so we shift the weights of the momenta by $2$. To summarize, we assign the following weights 
%
\begin{equation}\label{eqweight}
\begin{aligned}
\wt x_0 = 0 , \quad \wt y_1 = 2 , \quad \wt x_2 = \wt y_2 = \wt z_2 = 1,\\
\wt p_{x_0} = 2 , \quad \wt p_{y_1} = 0 , \quad \wt p_{x_2} = \wt p_{y_2} = 1.
\end{aligned}
\end{equation}

\subsection{Stationary discs}

We recall that a holomorphic disc 
$f = (g,h)$, with $g,h \in \mathcal{O}(\Delta) \cap C(\overline{\Delta})$ 
is said to be {\em attached} to $M$ if $f(b \Delta) \subset M$. We recall that 
the complex tangent space of $M$ at $p$ is given by $T^c_p M = T_p M \cap i T_p M$ and denote the 
conormal bundle of $M$ by $\mathcal{N}M \subset  T^* M $, defined by 
\[ \mathcal{N}_p M = \left\{ \theta_p \in T^*_p M \colon \theta_p (X_p ) = 0, \, X_p \in T_p^c M \right\}.  \]  An attached
disc $f=(g,h)$ is said to be {\em stationary} if it has a lift $(f,\tilde f)$ attached to $\mathcal{N} M$  which
is holomorphic up to a pole of order at most $1$ in $\Delta$. If $M=\partial \Omega$ 
is the boundary of a strictly pseudoconvex domain, stationarity is related to the Euler-Lagrange equations 
for extremal discs for the Kobayashi metric. 

In term of equations, we can use the fact that $\mathcal{N}_p M$ is spanned by $\varrho_z = (\dopt{\rho}{z_1},\dopt{\rho}{z_2})$ to 
express the fact that $f$ is stationary in the following form: $f$ is stationary if and only if  
there exists a real-valued positive function $a$ on $b\Delta$ such 
that the map $\tilde f = (\tilde g, \tilde h)$ defined by 
\begin{equation}\label{e:stat0} \begin{aligned}
\tilde g (\zeta) & = \zeta a (\zeta) \dopt{\rho}{z_1} \left( f(\zeta) , \overline{ f(\zeta )} \right) \\
\tilde h (\zeta) & = \zeta a (\zeta) \dopt{\rho}{z_2} \left( f(\zeta) , \overline{ f(\zeta )} \right),  
\end{aligned} 
\end{equation}
for  $\zeta \in b\Delta$, extends holomorphically to $\Delta$. 
To deal with this extension property, we will use  the well known fact (see \cite{be-de-la} for a proof) that a continuous function $\varphi:b \Omega\to\mathbb C$ defined on the smooth boundary of a simply connected domain $\Omega$ extends holomorphically to $\Omega$ if and only if it satisfies the \emph{moment conditions}
\begin{equation}\label{eqmom}
\int_{b\Omega} \zeta^m \varphi (\zeta) d\zeta =0 \ \ \mbox{ for all } m\geq 0.
\end{equation}

It turns out that if $M$ is strictly pseudoconvex, then its conormal bundle is 
actually totally real  \cite{webster}, and so the attachment of stationary discs turns into a 
standard Riemann-Hilbert problem \cite{lempert, forstneric, globevnik1, globevnik2}. In the case 
of the model hypersurface $2\real z_1 = |z_2|^2$, a typical stationary 
disc $f$ passing through $0$ at $1$ (i.e. $f(1)=0$) is $f(\zeta) = (1-\zeta, 1-\zeta)$ and its lifts are given by $(1-\zeta, 1-\zeta, a \zeta , a (\zeta - 1)), a \in \R$. 

The boundary traces of stationary discs are preserved under 
(local) CR diffeomorphisms in the following sense. In 
the case of a strictly pseudoconvex hypersurface $M$, every CR function on $M$
extends to the pseudoconvex side of $M$. Since the components of a CR map $H$ are 
CR functions, the map actually extends as a holomorphic map to the pseudoconvex side of $M$. Therefore, for
a small enough stationary disc attached to $M$, the disc $H\circ f$ is attached
to $H(M)$ and is stationary (this is obvious from the characterization as lifts, or one can use the defining equation
$\tilde \rho = \rho \circ H^{-1}$ in \eqref{e:stat0}). 

\section{The Fefferman Hamiltonian in normal form}\label{sec:Ham}
\subsection{The model case}

Consider a strictly pseudoconvex hypersurface of the form 
$$M=\{2 \real z_1=Q(z_2,\overline{z_2})\}\subset \C^2.$$
Such hypersurfaces (whose defining equations do not depend on $\imag z_1$) are 
called {\em rigid}; the $1$-parameter group of transformations 
$z_1 \mapsto z_1 + it$, $t\in \R$, yields a cyclic variable for the 
Hamiltonian \eqref{eqH}. 
 
 As before, we write $z_j=x_j+iy_j$, $j=1,2$ and use $(x_0,y_1,x_2, y_2)$ as
 variables on $\mathbb{S}^1\times M$. We also write
$z_0=e^{i\theta}$ and $x_0=\theta$.     
We now consider  the defining equation 
$$\rho=2 \real z_1-Q(z_2,\overline{z_2}).$$
 {
 We have 
 $$ \Phi=
  \det \left(\begin{matrix}
 \rho & 1 &   -Q_{\overline z_2}     \\
 1 & 0& 0 \\
 -Q_{ z_2} &  0 &  -Q_{z_2 \overline z_2}  \\
 \end{matrix}\right)= Q_{z_2 \overline z_2}$$
and
 \begin{eqnarray*}
 A^{-1}&=&  \left(\begin{matrix}
 0 & i&  -iQ_{\overline z_2}     \\
  -i & 0 & 0 \\
 iQ_{ z_2} &  0 &  -3 Q_{z_2 \overline z_2}  \\
 \end{matrix}\right)^{-1}
 =
  \left(\begin{matrix}
 0 & i&  0     \\
  -i & -\frac{|Q_{\overline z_2}|^2}{3 Q_{z_2 \overline z_2}}  & -\frac{Q_{\overline z_2}}{3 Q_{z_2 \overline z_2}}  \\
 0 &  -\frac{Q_{z_2}}{3 Q_{z_2 \overline z_2}}  &  -\frac{1}{3 Q_{z_2 \overline z_2}}  \\
 \end{matrix}\right).
 \end{eqnarray*}
} 
{Moreover following \eqref{eqH}, we have 
 $$\sum_{l\geq 1, k\geq 0}\Phi_{\overline z_l}A^{lk}(p_{x_k}-ip_{y_k})=\sum_{k\geq 0}\Phi_{\overline z_2}A^{2k}(p_{x_k}-ip_{y_k})$$
 and since $Q$ is independent of $z_1$ we also have 
 \begin{eqnarray*}
 \tilde{\Phi} A^{-1}&=&
 \left(\begin{matrix}
 0 & 0&  0   \\
 0& \frac{(n+1)}{\Phi}\Phi_{z_1\overline z_1}-\frac{2n+1}{\Phi^2}\Phi_{z_1}\Phi_{\overline z_1}  & \frac{(n+1)}{\Phi}\Phi_{z_1\overline z_2}-\frac{2n+1}{\Phi^2}\Phi_{z_1}\Phi_{\overline z_2} \\
 0& \frac{(n+1)}{\Phi}\Phi_{z_2\overline z_1}-\frac{2n+1}{\Phi^2}\Phi_{z_2}\Phi_{\overline z_1}  &  
 \frac{(n+1)}{\Phi}\Phi_{z_2\overline z_2}-\frac{2n+1}{\Phi^2}\Phi_{z_2}\Phi_{\overline z_2} \\
 \end{matrix}\right)\\
 \
 &=&\left(\begin{matrix}
 0 & 0&  0   \\
 0 & 0&  0   \\
 0& 0  &  
 \frac{3}{\Phi}\Phi_{z_2\overline z_2}-\frac{5}{\Phi^{{2}}}\Phi_{z_2}\Phi_{\overline z_2} \\
 \end{matrix}\right)\\
 \end{eqnarray*}
 }
 Thus the Fefferman Hamiltonian  is computed to be 
\begin{equation}\label{eqHrigid}
\begin{aligned}
H = & \  P A^{-1} P^*
-\frac{2}{3}\frac{p_{x_0}}{(Q_{z_2 \overline z_2})^2}
\imag \left(Q_{z_2 \overline z_2^2}\left(iQ_{z_2}p_{y_1}{-}p_{x_2}{+}ip_{y_2}\right)\right)\\
 & {+}\frac{p_{x_0}^2}{{6}(Q_{z_2 \overline z_2})^2}\left(3Q_{z_2^2 \overline z_2^2}-5\frac{Q_{z_2^2 \overline z_2}Q_{z_2 \overline z_2^2}}{Q_{z_2 \overline z_2}}\right).
\end{aligned}
\end{equation}
 Using the notations $p_{z_2} = p_{x_2}+ip_{y_2}$ and $p_{\overline z_2} = \overline{p_{z_2}} = p_{x_2}-ip_{y_2}$, the expression involving $A$ in \eqref{eqHrigid} is 
 explicitly given by 
\[ \begin{aligned} P A^{-1} P^*
  &=(p_{x_0},ip_{y_1},p_{z_2})
    \begin{pmatrix}
       0 & i&  0     \\
  -i & -\frac{|Q_{\overline z_2}|^2}{3 Q_{z_2 \overline z_2}}  & -\frac{Q_{\overline z_2}}{3 Q_{z_2 \overline z_2}}  \\
 0 &  -\frac{Q_{z_2}}{3 Q_{z_2 \overline z_2}}  &  -\frac{1}{3 Q_{z_2 \overline z_2}}  
    \end{pmatrix}
    \begin{pmatrix}
      p_{x_0} \\ - ip_{y_1} \\ p_{\bar z_2}
    \end{pmatrix}\\
&=
    2p_{x_0}p_{y_1} -\frac{|Q_{\overline z_2}|^2}{3 Q_{z_2 \overline z_2}} p_{y_1}^2 + i\frac{Q_{z_2}}{3 Q_{z_2 \overline z_2}}p_{y_1}p_{z_2}-i \frac{Q_{\overline z_2}}{3 Q_{z_2 \overline z_2}} p_{y_1}p_{\overline z_2}  - \frac{1}{3 Q_{z_2 \overline z_2}} |p_{z_2}|^2.
  \end{aligned}\]
Adding the rest of the Hamiltonian, we get
\[\begin{aligned}
H=& \left( \frac{1}{2}\frac{Q_{z_2^2 \overline z_2^2}}{Q^2_{z_2 \overline z_2}}-\frac{5}{6}\frac{Q_{z_2^2 \overline z_2}Q_{z_2 \overline z_2^2}}{Q^3_{z_2 \overline z_2}} \right)p^2_{x_0}  + \left(2 - \frac{Q_{z_2\overline z_2^2}Q_{z_2} +Q_{z_2^2\overline z_2} Q_{\overline z_2}}{3Q^2_{z_2\overline z_2}}\right) p_{x_0}p_{y_1}\\
&+i \frac{Q_{z^2_2\overline z_2}}{3Q^2_{z_2\overline z_2}} p_{x_0}p_{z_2} - i \frac{Q_{z_2\overline z_2^2}}{3Q^2_{z_2\overline z_2}} p_{x_0}p_{\overline z_2} -\frac{|Q_{\overline z_2}|^2}{3 Q_{z_2 \overline z_2}} p_{y_1}^2 + i\frac{Q_{z_2}}{3 Q_{z_2 \overline z_2}}p_{y_1}p_{z_2}
\\
&-i \frac{Q_{\overline z_2}}{3 Q_{z_2 \overline z_2}} p_{y_1}p_{\overline z_2} - \frac{1}{3 Q_{z_2 \overline z_2}} |p_{z_2}|^2
\\ \\
=& P  \left(\begin{matrix}
\frac{1}{2}\frac{Q_{z_2^2 \overline z_2^2}}{Q^2_{z_2 \overline z_2}}-\frac{5}{6}\frac{Q_{z_2^2 \overline z_2}Q_{z_2 \overline z_2^2}}{Q^3_{z_2 \overline z_2}} & i\left(1- \frac{Q_{z_2\overline z_2^2}Q_{z_2} }{3Q^2_{z_2\overline z_2}}\right)&  - i \frac{Q_{z_2\overline z_2^2}}{3Q^2_{z_2\overline z_2}}     \\
-i\left(1-  \frac{Q_{z_2^2\overline z_2} Q_{\overline z_2}}{3Q^2_{z_2\overline z_2}}\right) & -\frac{|Q_{\overline z_2}|^2}{3 Q_{z_2 \overline z_2}}   & -\frac{Q_{\overline z_2}}{3 Q_{z_2 \overline z_2}}  \\
i \frac{Q_{z^2_2\overline z_2}}{3Q^2_{z_2\overline z_2}} &  -\frac{Q_{z_2}}{3 Q_{z_2 \overline z_2}}  &  - \frac{1}{3 Q_{z_2 \overline z_2}}  \\
\end{matrix}\right) P^*.
\end{aligned}\]

As a particular case we consider 
$$Q(z_2,\overline z_2)= |z_2|^2+az_2^2\overline{z_2}^4+\overline{a}z_2^4\overline{z_2}^2 = 
|z|^2 + 2  a |z_2|^4\real z_2^2,  $$ where we assume that $a \in \R$ from now on, we can
explicitly compute that
\[\begin{aligned} H &= -\frac{2}{3d^2}  2  p_{x_0} p_{y_1} \left(8 a z_2 \bar{z}_2 \left(5 a z_2^5 \bar{z}_2+14 a
                      z_2^3 \bar{z}_2^3+5 a z_2 \bar{z}_2^5+2 \bar{z}_2^2+2 z_2^2\right)-3 d^2\right)\\
   &\quad -\frac{4 i a}{d^3} p_{x_0} 
   \left(z_2 \left(3 \bar{z}_2^2+z_2^2\right) p_{\bar z_2}-\left(\bar{z}_2^3+3 z_2^2
     \bar{z}_2\right) p_{z_2}\right)\\
  &\quad -\frac{1}{3d}  \left|z_2 p_{y_1} \left(2 a z_2 \bar{z}_2
   \left(2 \bar{z}_2^2+z_2^2\right)+1\right)-i p_{z_2}\right|^2
  \\
 &\quad  -\frac{4 a}{3d^3} p_{x_0}^2 \left(40 a z_2 \bar{z}_2 \left(10 z_2^2
   \bar{z}_2^2+3 \bar{z}_2^4+3 z_2^4\right)-9 d \left(\bar{z}_2^2+z_2^2\right)\right), \end{aligned}\]
where $d=Q_{z_2 \bar z_2} = 8 a z_2^3 \bar{z}_2+8 a z_2 \bar{z}_2^3+1$. If we 
truncate this expression at order 6, discarding terms which are quadratic or higher order in $a$, we obtain
\[\begin{aligned}
H_0=& 24  a (\real z_2^2)  p^2_{x_0} + \left(2- \frac{64 a}{3}\left( |z_2|^2 \real z_2^2 \right) \right) p_{x_0}p_{y_1}
+\frac{i a}{3} \left(24  z_2^2 \overline z_2+ 8\overline z_2^3\right) p_{x_0}p_{z_2} \\&- \frac{i a}{3} \left(8 z_2^3 + 24  z_2 \overline z_2^2\right) p_{x_0}p_{\overline z_2} 
+\frac{1}{3} \left(-|z_2|^2+4 a |z_2|^4 \real z_2^2\right) p_{y_1}^2 \\
&
+ \frac{i}{3} \left(\overline z_2 -  a z_2 \left( 4 |z_2|^4 +  6  \overline z_2^4\right)\right)p_{y_1}p_{z_2} - \frac{i}{3} \left(z_2 - a \overline{z}_2 \left( 4 |z_2|^4 +  6 z_2^4\right)\right) p_{y_1}p_{\overline z_2} \\
 & - \frac{1}{3}\left(1 - 16 a |z_2|^2 \real z_2^2 \right) |p_{z_2}|^2.
\end{aligned}\]
We are going to use this as our {\em model Hamiltonian}. 
The system of ODEs associated to $H_0$  can now be obtained as in Example \ref{exsphere} with the only difference that $z_2' = 2\frac{\partial H_0}{\partial p_{\overline z_2}}$, $\overline z_2' = 2\frac{\partial H_0}{\partial p_{z_2}}$, $p_{z_2}' = -2\frac{\partial H_0}{\partial \overline z_2}$, and  $p_{\overline z_2}' = -2\frac{\partial H_0}{\partial z_2}$:
\begin{equation}\label{modelODE}
\begin{aligned}
0  = & \ H_0\left(z_2(t), p_{x_0}(t), p_{y_1}(t), p_{z_2}(t)\right)\\
x_0'  = & \  48 a \left( \real z_2 \right) p_{x_0} + \left(2- \frac{64 a}{3}\left(  |z_2|^2 \real z_2 \right)\right) p_{y_1}\\
&  +\frac{i a}{3} \left(24  z_2^2 \overline z_2 + 8 \overline z_2^3\right) p_{z_2} - \frac{i a}{3} \left(8 z_2^3 + 24  z_2 \overline z_2^2\right) p_{\overline z_2}\\
y_1' = &   \left(2- \frac{64 a}{3}( |z_2|^2 \real z_2^2)\right) p_{x_0} + \frac{2}{3} \left(-|z_2|^2+2 a |z_2|^4\real z_2^2 \right) p_{y_1}   \\
&  + \frac{i}{3} \left(\overline z_2 - a z_2 (4 |z_2|^4 +  6  \overline z_2^4)\right)p_{z_2}- \frac{i}{3}\left(z_2 + a \overline z_2 (4 |z_2|^4 +  6   z_2^4)\right) p_{\overline z_2} \\
z_2' = & -\frac{16 i a}{3}\left(  z_2^3 + 3  |z_2|^2 \overline z_2\right)p_{x_0} - \frac{2i}{3}\left(z_2 -  a |z|^2( 6 z_2^3 - 4 z_2  \overline z_2^2 ) \right)p_{y_1} \\
& - \frac{2}{3}\left(1 - 16 a |z_2|^2 \real z_2^2 \right)p_{z_2}\\
p_{x_0}' = & \ 0\\
p_{y_1}' = & \ 0\\
p_{z_2}' = & -48 a \overline z_2p_{x_0}^2 + \frac{64 a}{3}\left(  z_2^3 + 3  z_2 \overline z_2^2\right) p_{x_0}p_{y_1}  - \frac{96 a i}{3}\left(\real z_2^2\right) p_{x_0}p_{z_2} \\
& + \frac{96 a i}{3}|z_2|^2 p_{x_0}p_{\overline z_2} + \frac{2}{3}\left(z_2 - 4 a |z_2|^2 ( z_2^3 + 2  z_2 \overline z_2^2)\right)p_{y_1}^2 +  \\
&  + \frac{i}{3}\left(-2 + 16 a |z_2|^2 (z_2^2 +  3 \overline z_2^2)\right)p_{y_1}p_{z_2} - \frac{12 a i}{3}\left( z_2^4 + 2 z_2^2\overline z_2^2\right) p_{y_1}p_{\overline z_2}  \\
& - \frac{16 a}{3}\left( z_2^3  + 3 z_2 \overline z_2^2\right) |p_{z_2}|^2\\
\end{aligned}
\end{equation}

\subsection{The general case: weighted Taylor expansion of the Fefferman Hamiltonian}

Consider a strictly 
pseudoconvex hypersurface $M\subset \C^2$, locally written in Chern-Moser normal form as $\rho=0$, with
\[\rho= 2\real z_1 - (|z_2|^2 +2 a |z_2|^4 \real z_2^2 + (\imag z_1)\cdot \eta(\imag z_1, z_2, \overline z_2) + \delta(z_2,\overline z_2))\]
where $a\in \R$ and $\eta$, $\delta$ are functions of weighted orders $O(6)$ and $O(7)$ respectively. In the following, we give weighted degrees to all the monomials in the variables $x_0, y_1, z_2$ and the conjugate momenta $p_{x_0}, p_{y_1}, p_{z_2}$ according to the weight assignments \eqref{eqweight}.

\begin{lemma}\label{lemrho}
Let $H$ and $H_0$  be the Fefferman Hamiltonians associated respectively to $\rho$ and  $\rho_0 = 2\real z_1 - (|z_2|^2 +az_2^2\overline z_2^4 + a z_2^4\overline z_2^2) $. We then have  
$$H=H_0+O(7).$$ 
\end{lemma}
\begin{proof}
We denote by $Q_k$ a generic homogeneous polynomial of weighted order $k$. We emphasize that, in the below, the polynomials $Q_k$ are not necessarily the same.  Moreover,  throughout the computations, each $Q_k$, $k\leq 6$, comes directly from differentiating or multiplying terms in $\rho_0$.     
We  then 
write  $$\rho =  \rho_0 + O(7) = 2\real z_1-|z_2|^2 -2 a |z_2|^4 \real z_2^2 + O(7).$$ 
The expression of the Hamiltonian $H$ associated to $\rho$ is given by \eqref{eqH}. Since all terms in $H$ involve the matrix  $A^{-1}$, we first focus on computing the order in its entries. Computing explicitly the inverse of $A$, we get
\[A^{-1}= \left(\begin{matrix}
0 & i\rho_{\overline z_1}&  i \rho_{\overline z_2}     \\
-i \rho_{ z_1}& 3 \rho_{z_1 \overline z_1} & 3 \rho_{z_1 \overline z_2} \\
-i \rho_{ z_2} & 3 \rho_{z_2 \overline z_1} & 3  \rho_{z_2 \overline z_2}  \\
\end{matrix}\right)^{-1}
\]
\[= \frac{3}{\det A}\left(\begin{matrix}
3(\rho_{z_1 \overline z_1}  \rho_{z_2 \overline z_2} -  \rho_{z_1 \overline z_2} \rho_{z_2 \overline z_1})  & -i(\rho_{\overline z_1}\rho_{z_2 \overline z_2} - \rho_{\overline z_2} \rho_{z_1 \overline z_2} ) &  i(\rho_{\overline z_1}\rho_{z_1 \overline z_2} - \rho_{\overline z_2} \rho_{z_1 \overline z_1} )    \\
i(\rho_{ z_1}\rho_{z_2 \overline z_2} - \rho_{ z_2} \rho_{z_1 \overline z_2} ) &\rho_{ z_2}\rho_{\overline z_2}/3   & - \rho_{ z_2}\rho_{\overline z_1}/3 \\
-i(\rho_{ z_1}\rho_{z_2 \overline z_1} - \rho_{ z_2} \rho_{z_1 \overline z_1} ) & - \rho_{ z_1}\rho_{\overline z_2}/3   & \rho_{ z_1}\rho_{\overline z_1}/3  \\
\end{matrix}\right)
\]

A careful bookkeeping of the weighted order of the entries of the matrix above as well as the tracking of the contribution of $\rho_0$ alone lead to  
\begin{equation}\label{eqainv}
A^{-1}= \frac{3}{\det A}\left(\begin{matrix}
O(4) &Q_0 + Q_4 + O(5)&  O(5)  \\
Q_0 + Q_4 + O(5)&Q_2+Q_6+O(7)  & Q_1 + Q_5 + O(6) \\
O(5)& Q_1+Q_5+O(6) & Q_0 + O(6) \\
\end{matrix}\right)
\end{equation}
Moreover we have $\displaystyle \frac{1}{\det A}=\frac{1}{ 3\rho_{z_2\overline z_2} + O(5)}=-\frac{1}{3}+Q_4+O(5)$.

We can now investigate the first term $P A^{-1} P^*$ in the Hamiltonian Hamiltonian $H$. It follows from \eqref{eqainv} and the order of the components of $P=(p_{x_0},ip_{y_1},p_{x_2}+ip_{y_2})$ (see (\ref{eqweight})) that
$$P A^{-1} P^*=Q_2+Q_6+O(7).$$ 
We now consider the term 
$\displaystyle \frac{2p_{x_0}}{\Phi}\imag \left( \overline{\partial}\Phi \cdot A^{-1}\cdot P^*\right)$.
By a very similar computation, we obtain
$$\overline{\partial}\Phi=\left(0,\Phi_{\overline z_1},\Phi_{\overline{z}_2}\right)=\left(0,O(4),Q_3+O(4)\right),$$
and thus
$$\frac{2p_{x_0}}{\Phi}\imag \left( \overline{\partial}\Phi \cdot A^{-1}\cdot P^*\right)=Q_6+O(7).$$
Finally, we also have 
$$\frac{p_{x_0}^2}{2\Phi} \Tr (\tilde{\Phi} A^{-1})=Q_6+O(7).$$
This proves the lemma. 
\end{proof}

\section{Proof of the main theorem}\label{sec:proof}
Let us reformulate our main theorem in the way we will prove it. 
\begin{theorem} \label{thmmain} 
Let $M\subset (\C^2,0)$ be a strictly pseudoconvex hypersurface of class $\mathcal{C}^{12}$ with  local defining equation of the form
$$\rho= 2\real z_1 - \left(|z_2|^2 +2 a |z_2|^4 \real  z_2^2 + (\imag z_1)\cdot \eta(\imag z_1, z_2, \overline z_2) + \delta(z_2,\overline z_2)\right),$$
where $\eta$ and $\delta$ are of weighted order $O(6)$ and $O(7)$ respectively. 
 If every chain for $M$ 
for a family of starting conditions as in Lemma \ref{leminit} is the boundary of a stationary disc then $a = 0$.
\end{theorem}

In order to prove Theorem~\ref{thmmain}, we compute an asymptotic expansion of a family of chains which would
come from circles in the case of the sphere, and find that the moment conditions for the members of the family
yield an obstruction to umbilicity in the fourth order term of that expansion. The details are as follows. 

\subsection{Computations of the orbits}
We will use a special family of solutions of the Hamiltonian system associated to $\rho$ depending on a small real parameter $s>0$. This will be achieved by making a suitable choice of initial conditions imposed in order to reproduce the circular orbits in the case of the sphere $\{2\real z_1 = |z_2|^2\}$.  
\begin{lemma}\label{leminit}
Let $x_0(.,0), \varphi, \psi, \xi$ be four functions in $s$ of  class $\mathcal{C}^7$. Then there exists a family of initial conditions of the form 
\begin{equation}\label{eqinit}
\begin{aligned}
y_1(s,0) & = s^2\varphi(s)\\
z_2(s,0)& = s + s^5\psi(s)\\
p_{x_0}(s,0)& = -\frac{1}{2}s^2+s^6\chi(s)\\
p_{y_1}(s,0)& = -\frac{3}{4}\\
p_{z_2}(s,0)& =  -\frac{3i}{4}s + s^5\xi(s)\\
\end{aligned}
\end{equation}
for  some function $\chi$  of  class $\mathcal{C}^7$ such that 
\begin{equation}\label{eqH0}
H(x_0(s,0), y_1(s,0),z_2(s,0), p_{x_0}(s,0), p_{y_1}(s,0), p_{z_2}(s,0)) = 0,
\end{equation}
where $H$ is the Fefferman Hamiltonian associated to $\rho$. 
\end{lemma}
\begin{proof}
Substituting the initial conditions \eqref{eqinit} 
into the formula of the Hamiltonian and using Lemma \ref{lemrho}, we get
\[(24a  s^2 + O(s^6))\left(\frac{1}{4}s^4-s^8\chi+s^{12}\chi^2 \right) +\]
\[\left(2 - \frac{64}{3}a  s^4 + O(s^8) \right)\left(\frac{3}{8}s^2 - \frac{3}{4}s^6\chi \right)+\left(16a  s^4 + O(s^8)\right)\left(-\frac{1}{2}s^2 +s^6\chi \right) +\]
\[ \frac{3}{16}(-s^2 + O(s^6))- \frac{3}{8}\left(s^2 + O(s^6)\right) - \frac{3}{16}s^2 + O(s^6) = 0 \]

where the $O(\cdot)$ terms in the expression above may depend on $s,\alpha, x_0(.,0), \varphi, \psi, \xi$ but not $\chi$. Developing the products explicitly, we note that the $s^2$ terms (coming only from the spherical part of the Hamiltonian) simplify, while the next lowest order terms are $O(s^6)$, giving the following:
\[s^6\left(\left(-\frac{3}{2} + O(s)\right)\chi + O(s^{6})\chi^2  + O(1)\right) = 0\]
where once again the $O(\cdot)$ terms do not depend on $\chi$. Applying the implicit function theorem to the expression inside the parenthesis, we conclude that for any choice of $x_0(.,0)$, $\varphi$, $\psi$, and $ \xi$ there exists locally a unique function $\chi(s)$  of  class 
$\mathcal{C}^7$ such that Equation \eqref{eqH0} is satisfied.
\end{proof}
This choice of initial conditions then provides the following  family of solutions of the Hamiltonian system associated to $\rho$
\begin{equation}\label{eqfam}
\begin{aligned}
x_0(s,t) &  =  x_0^0(t) + \ldots\\
y_1(s,t) & = s^2 y_1^2(t) + \ldots\\
z_2(s,t) & = s z^1_2(t) + s^2 z^2_2(t) + \dots\\
p_{x_0}(s,t) &  =-\frac{1}{2}s^2+s^6\chi(s) + \dots\\
{p_{y_1}(s,t)} & = -\frac{3}{4} + \dots\\
p_{z_2}(s,t) & = s p_{z_2}^1(t) + s^2 p_{z_2}^2(t) + \ldots\\
\end{aligned}
\end{equation}

\begin{remark}\label{rempara} 
For our later computations, the most important components of the solutions are  
$y_1(s,t)$  and $z_2(s,t)$.
 Due to the expression of the defining function, the terms involving $y_1(s,t)$ appear to high order, and for this reason, it is enough to know that $y_1(s,t)$ is of order $O(s^2)$. As for $z_2(s,t)$, we  need to know more precisely its asymptotic behavior, beyond the fact that its order is $O(s)$.
\end{remark}
\begin{lemma}\label{eqz_2}
We have 
\begin{equation*}
z_2(s,t)= s e^{it} -\frac{4}{3} s^5 {a} e^{3it} + O(s^6).
\end{equation*}
\end{lemma}
\begin{proof}

According to Lemma \ref{lemrho}, solving the Hamiltonian system associated to $\rho$ up to order $s^5$ is equivalent to solving the system \eqref{modelODE}. We will proceed
iteratively by expanding in powers of $s$. The terms of order $s$ in the equations for $z_2'$ and $p_{z_2}'$ give the system
\[ {z^1_2}'(t)=   \frac{i}{2} z^1_2(t)  - \frac{2}{3} p_{z_2}^1(t)\]
\[{p_{z_2}^1}'(t) =  \frac{3}{8} z_2^1(t)  +\frac{i}{2} p_{z_2}^1(t).\]
Using the initial conditions $z^1_2(0)=1$ and $p_{z_2}^1(0)=-\frac{3i}{4}$, we get $z_2^1(t)=e^{it}$ and $p_{z_2}^1(t)=-\frac{3i}{4}e^{it}$.
Now, the terms in $s^2$ lead to 
\[ {z^2_2}'(t)=   \frac{i}{2} z^2_2(t)  - \frac{2}{3} p_{z_2}^2(t)\]
\[{p_{z_2}^2}'(t) =  \frac{3}{8} z_2^2(t)  +\frac{i}{2} p_{z_2}^2(t).\]
If $p_{y_1}^1 = 0, z_2^2(0)=0$  and $p_{z_2}^2(0)$, then $z_2^2(t)\equiv 0$ and  $p_{z_2}^2(t)\equiv 0 $. Similarly, we get $z_2^j(t)\equiv 0$ and  $p_{z_2}^j(t)\equiv 0$ for $j=3,4$.  
Finally, computing the terms of order $s^5$, we obtain 
\[ {z^5_2}'(t)=   \frac{i}{2} z^5_2(t)  - \frac{2}{3} p_{z_2}^5(t)-\frac{13i}{3}
a e^{3it} + 2i a e^{-it} \]
\[{p_{z_2}^5}'(t) =  \frac{3}{8} z_2^5(t)  +\frac{i}{2} p_{z_2}^5(t) + \frac{17}{4} a e^{3it} + \frac{9}{2} a e^{-it}.\]
A particular solution is given by 
\[z_2^5(s,t)= -\frac{4}{3}a e^{3it}\]
\[p_{z_2}^5(s,t)=-\frac{3i}{2}a e^{3it} +3iae^{-it}.\]
This concludes the proof of the lemma. 
\end{proof}

\subsection{Enforcing stationarity}
We consider the $y_1(s,t), z_2(s,t)$ components of the family of solutions \eqref{eqfam} of the Hamiltonian system associated to $\rho$. Assume that there exists a family of stationary discs $f_s=(g_s,h_s)$ such that $f_s(b\Delta)$ coincides with the image of the chain $(z_1(s,.),z_2(s,.))$, where the real part of $z_1(s,.)$  is determined by $\rho$.  In particular, note this implies that the chain  $(z_1(s,.),z_2(s,.))$ is periodic. We denote by $T_s$ its period.  We may assume that the projection on the second coordinate $\pi_2: f_s(\Delta)\to \mathbb C_{z_2}$ is injective and that $0\in \pi_2(f_s(\Delta))=h_s(\Delta)$. Then we can take as $h_s$ the unique Riemann map $\Delta\to h_s(\Delta)$ such that $h_s(0)=0$ and $h'_s(0)>0$.

By definition, $f_s=(g_s,h_s)$ is stationary if and only if there exists a continuous function $a_s:b\Delta\to \mathbb R^+$ and  functions $\widetilde g_s,\widetilde h_s\in \mathcal O(\Delta)\cap C(\overline \Delta)$ satisfying
\begin{equation} \label{eq:stat}
\begin{aligned} 
\widetilde g_s(\zeta)&= \zeta a_s(\zeta) \frac{\partial \rho}{\partial z_1}\left(g_s(\zeta), h_s(\zeta), \overline{g_s(\zeta)},\overline{h_s(\zeta)}\right)  \\ 
\widetilde h_s(\zeta) &= \zeta a_s(\zeta) \frac{\partial \rho}{\partial z_2}\left(g_s(\zeta), h_s(\zeta), \overline{g_s(\zeta)},\overline{h_s(\zeta)}\right)
\end{aligned} 
\end{equation}
for all $\zeta\in b\Delta$.
Define now $\Omega_s:=h_s(\Delta)\subset \mathbb C$ and $S_s:=b\Omega_s$. We also write $h_s^{-1}(z)=ze^{\varphi_s(z)}$ for a certain holomorphic function 
$\varphi_s(z)$.
Evaluating Equations (\ref{eq:stat}) for $\zeta=h_s^{-1}(z)$, we obtain
\[
\begin{aligned} 
\widetilde g_s(h_s^{-1}(z)) &= z e^{\varphi_s(z)} a_s(h_s^{-1}(z)) \frac{\partial \rho}{\partial z_1}\left(g_s(h_s^{-1}(z)), z, \overline {g_s(h_s^{-1}(z))}, \overline{z} \right)\\
\widetilde h_s(h_s^{-1}(z)) &= z e^{\varphi_s(z)} a_s(h_s^{-1}(z)) \frac{\partial \rho}{\partial z_2}\left(g_s(h_s^{-1}(z)),z,\overline {g_s(h_s^{-1}(z))},  \overline{z}  \right) 
\end{aligned} 
\]
for all $z\in S_s$. We then set $b_s(z):=a_s(h_s^{-1}(z))$, $G_s(z):=e^{-\varphi_s(z)}\widetilde g_s(h_s^{-1}(z))$ and $H_s(z):=e^{-\varphi_s(z)}\widetilde h_s(h_s^{-1}(z)) $. Furthermore, if we write each disc $f_s(\Delta)$ as a graph $\{z_1=w_s(z)\}$ over its projection $\Omega_s$, we have $w_s(z)=g_s(h_s^{-1}(z))$. Thus we can rewrite the previous system as
\begin{equation} \label{eq:stat2}
\begin{aligned} 
G_s(z) &= z  b_s(z) \frac{\partial \rho}{\partial {z_1}}\left(w_s(z),z, \overline{w_s(z)}, \overline{z}\right) \\
H_s(z) &= z  b_s(z) \frac{\partial \rho}{\partial {z_2}}\left(w_s(z),z, \overline{w_s(z)}, \overline{z} \right) 
\end{aligned} 
\end{equation}
for $z\in S_s$. In order to apply the moment conditions \eqref{eqmom} to the functions $G_s$ and $H_s$, we need to find an adapted parametrization of the curve $S_s$.
We first consider the scaling $\Lambda_s: \mathbb{C}\to  \mathbb{C}$ defined by 
$\Lambda_s(z)=z/s$ and define $\widetilde \Omega_s:=\Lambda_s(\Omega_s)$ and $\widetilde S_s:=\Lambda_s(S_s)$. Note that, with this change of variables, the moment 
conditions \eqref{eqmom}  applied to $G_s(z)$ and $H_s(z)$ become 
\begin{equation}\label{eqmomscal}
\int_{\widetilde S_s} z^m G_s(sz) dz =\int_{\widetilde S_s} z^m H_s(sz) dz =0,
\end{equation}
for all  $m\geq 0.$
We may now set an adapted parametrization of $\widetilde S_s$. Since the image $f_s(b\Delta)$ coincides with the image of the chain $(z_1(s,.),z_2(s,.))$, 
we consider the parametrization of $\widetilde S_s$ given by 
\begin{equation}\label{eqpara}
[0,T_s]\ni t\mapsto \widehat z_2(s,t) :=\frac{z_2(s,t)}{s}\in \widetilde S_s.
\end{equation} 
According to Lemma \ref{eqz_2}, we can write $\widehat z_2(s,t) = r(s, t)e^{it}$ with 
$$r(s,t)= 1 + k(t) s^4 + O(s^5),$$ 
where 
$$k(t) = -\frac{4}{3} a e^{2it}.$$ Note that $r(s,t)$ is not necessarily real valued. Moreover, {due to standards results on ODEs (see for instance
\cite[Theorem 4.1]{MR1929104})}, the parametrization $\widehat z_2(s,t)$ is of class $\diffable{7}$ in both variables. 
A straightforward computation leads to the following useful lemma.
\begin{lemma}\label{lemr}
For $j\geq 1$,  
\begin{align*}
r(s, t)^j  & = 1 + j k(t) s^{4} + O(s^{5}), \\
r(s, t)^j \overline {r(s, t)}  & = 1 + \left(j k(t)+\overline{k(t)}\right) s^{4} + O(s^{5}), \\
\frac{\partial r}{\partial t}(s, t) & = \frac{dk}{dt}(t)s^4 + O(s^5).
\end{align*}
\end{lemma}
Moreover, in order to apply Fourier analysis later on, we need to understand the behavior of the period $T_s$ of $\widehat z_2(s,\cdot)$ as $s\to 0$.  
\begin{lemma}\label{lemper}
The function $[0,\epsilon]\ni s\mapsto T_s\in \mathbb R$ is of class $\mathcal{C}^7$. Furthermore, we have 
$$T_s=2\pi + O(s^5).$$
\end{lemma}
\begin{proof}

The function $(s,t) \mapsto \widehat z_2(s,t)$ of class $\mathcal{C}^7$ and, 
as $s\to 0$, converges uniformly to $t \mapsto e^{it}$ on any fixed 
neighborhood of $[0,2\pi]$. Since by assumption, $\widehat z_2(s,t)$ 
parametrizes a simple closed curve on $[0,T_s]$, the period $T_s$ tends to $2\pi$ as 
$s\to 0$.

By the $C^k$ smoothness of $\widehat z_2(s,t)$, we have $|\widehat z_2'(s,t)|\to |\widehat z_2'(0,t)| =1$ uniformly as $s\to 0$. Then the period $T_s$ must satisfy $\int_0^{T_s}|\widehat z_2'(s,t)|dt = \mbox{length} (\widetilde S_s) \to 2\pi$ as $s\to 0$. This is only possible if $T_s\to 2\pi$ as $s\to 0$, since any sequence $s_n\to 0$ such that $|T_{s_n}-2\pi|>\varepsilon>0$ would lead to a contradiction by taking the limit as $n\to\infty$.

To prove the smoothness of $T_s$, we will use the fact that the function $\psi$ appearing in Lemma \ref{leminit} {is real-valued (since $a\in \mathbb R$)}, which implies $\imag \widehat z_2(s,0)=0$ for $s\in [0,\epsilon]$. Consider the function $\iota(s,t) := \imag \widehat z_2(s,t)$. At $s=0$ and $t=2\pi$, we have
\[\iota(0,2\pi) = \imag \widehat z_2(0,2\pi) = 0, \ \ \frac{\partial \iota}{\partial t}(0,2\pi) = \imag \frac{\partial}{\partial t}(\widehat z_2(0,2\pi)) = 1. \]
By the implicit function theorem there exists a function $\kappa:[0,\epsilon]\to\mathbb R $, of the same smoothness as $\iota$, such that $\kappa(0)=2\pi$ and $\iota(s,\kappa(s))=0$. We claim that $T_s=\kappa(s)$. Indeed, by the implicit function theorem, $t=\kappa(s)$ represent the unique time in a neighborhood of $t=2\pi$ at which the curve $\widehat z_2(s,t)$ crosses the line $\imag z=0$. Since the period of $\widehat z_2(s,\cdot)$ approaches $2\pi$ as $s\to 0$, and $\imag \widehat z_2(s,0) = 0$, we necessarily have $\widehat z_2(s,\kappa(s))=\widehat z_2(s,0)$, which means that $T_s=\kappa(s)$ is of class $\diffable{7}$ near $s=0$.

We now turn to the asymptotic expression of $T_s$. By Lemma \ref{eqz_2}
\[\widehat z_2(s,t)=  e^{it} -\frac{4}{3} s^4 a e^{3it}  + O(s^5)\]
so that $\widehat z_2(s,t)$ can be seen as a small perturbation of the unit circle parametrized by $t\mapsto e^{it}$.  

Denoting the velocity vector of $\widehat z_2$ by $\widehat z_2 ' = \frac{\partial \widehat z_2}{\partial t}$, we have 
$$\widehat z_2 '(s,t)=ie^{it}+ O(s^4).$$ On the one hand, from the expression of $\widehat z_2(s,t)$, we have
\[\widehat z_2(s,2\pi) = \widehat z_2(s,0) + O(s^5).\]
On the other hand, we can write
\[\widehat z_2(s,2\pi) = \widehat z_2(s,T_s)+\int_{T_s}^{2\pi} \widehat z_2'(s,t)dt \]
and since $\widehat z_2(s,T_s) = \widehat z_2(s,0)$, putting together the expressions above we get
\[O(s^5)=\int_{T_s}^{2\pi}\widehat z_2'(s,t)dt = \int_{T_s}^{2\pi} \left(ie^{it} + O(s^4)\right) dt.  \]
For $s$ small enough we may suppose that $|ie^{it}+O(s^4)|\geq 1/2 $ and moreover that 
$$|\arg (ie^{it} +O(s^4))-\pi/2|<\pi/4$$
 for $t\in[T_s,2\pi]$ due to the fact that $T_s$ is close to $2\pi$. It follows that 
 $$\imag (ie^{it}+O(s^4))\geq \frac{\sqrt{2}}{4}$$ for $t\in [T_s,2\pi]$ and thus
\[ O(s^5) = \int_{T_s}^{2\pi} \imag \left(ie^{it} + O(s^4)\right) dt \geq \frac{\sqrt{2}}{4}|T_s-2\pi|\]
so that $|T_s-2\pi|=O(s^5)$.

\end{proof}

In view of Equation \eqref{eqmomscal}, we now  define 
$$c(s, t):=b_s(s\widehat z_2(s,t))=a_s(h_s^{-1}(s\widehat z_2(s,t))=a_s(h_s^{-1}(z_2(s,t)).$$

\begin{lemma}\label{lemc}
There is a choice of $a_s$ such that the function $c(s, t)$ is of class $\diffable{4}$ in a neighborhood of $\{0\}\times [0,2\pi]$ and satisfies 
$\int_0^{2\pi} c(s, t) dt =1$ for all $s>0$ small enough.
\end{lemma}
The proof is an adaptation of both proofs of Lemma 3.3 and Lemma 3.4 in \cite{be-de-la}. The main differences come from the facts 
that, in the present paper, the parametrization intervals depend on $s$, and the  first component of the discs we consider  is not constant. 

\begin{proof}

We first show that $h_s^{-1}(z_2(s,t))$ is of class $\mathcal{C}^{5}$ in both variables $s$ and $t$. In order to extend the parametrization \eqref{eqpara} to a uniform domain, namely the unit disc, we consider 
$$ [0,2\pi]\ni t\mapsto \widehat z_2\left(s,\frac{T_s}{2\pi}t\right)\in \widetilde S_s.$$
According to Lemma \ref{lemper}, this map is of class $\mathcal{C}^7$ in both variables $s$ and $t$. Moreover, its form allows us to extend it to the interior of the unit disc, and, thus,  to obtain a family of diffeomorphisms $\Gamma_s=\Gamma(s,\cdot):\overline{\Delta} \to \overline{\widetilde \Omega_s}$ of class  $\mathcal{C}^{7}$ in both variables  $s$ and $z$. It follows from Corollary 9.4 in \cite{MR3118395} and Lemma 2.1 in \cite{be-de-la} that  the function $(s,t) \mapsto h_s^{-1}\left(s\widehat z_2\left(s,\frac{T_s}{2\pi}t\right)\right)$, and so $(s,t) \mapsto h_s^{-1}(z_2(s,t))$, are of class  $\mathcal{C}^{5}$.
 
We now define 
$\widehat{a_s}$,   by 
\begin{equation*}
\frac{1}{\widehat{a_s}(\zeta)}=\zeta\partial \rho (f_s(\zeta))\cdot f_s'(\zeta),
\end{equation*}
$\zeta \in b\Delta$, where $\cdot$ denotes the dot product in $\C^2$. According to Pang \cite{MR1250257}, since $f_s$  is stationary and satisfies  (\ref{eq:stat}) for a continuous function $a_s$ 
then  $a_s$ is a positive multiple of $\widehat{a_s}$, where the multiple may be any function of $s$. 
We now show that the function $(s,t) \mapsto \widehat{a_s}(h_s^{-1}(z_2(s,t))$ is of class $\mathcal{C}^4$. Note first that the map 
 $$\partial \rho \left(f_s(h_s^{-1}(z_2(s,t)e^{it})\right)=
\partial\rho \left(w_s(z_2(s,t)), z_2(s,t), \overline{w_s(z_2(s,t))}, \overline{z_2(s,t)}\right)$$ 
is of class $\diffable{7}$ in both variables.
To study the smoothness of $f_s' (h_s^{-1}(z_2(s,t))$, note that by the chain rule, we have 
\begin{eqnarray*}
\frac{d}{dt} \left(z_1(s,t), z_2(s,t))\right)&=& \frac{d}{dt} f_s (h_s^{-1}(z_2(s,t)) \\
&=& f_s' (h_s^{-1}(z_2(s,t)) \cdot \frac{d}{dt}h_s^{-1}(z_2(s,t)),
\end{eqnarray*}
and so 
$$f_s' (h_s^{-1}(z_2(s,t))=\left(\frac{\frac{d}{dt}\left(z_1(s,t))\right)}
{\frac{d}{dt}h_s^{-1}(z_2(s,t))},
\frac{\frac{d}{dt}\left(z_2(s,t))\right)}{\frac{d}{dt}h_s^{-1}(z_2(s,t))}
\right).
$$
Following the proof of Lemma 3.3 in \cite{be-de-la}, we have $h_s^{-1}(z_2(s,t))=e^{it}+O(s)$, and since $h_s^{-1}(z_2(s,t))$ is of class $\mathcal{C}^{5}$, so is the map $f_s' (h_s^{-1}(z_2(s,t))$. This shows that the function 
 $(s,t) \mapsto \widehat{a_s}(h_s^{-1}(z_2(s,t))$   is of class $\mathcal{C}^{4}$.

Finally, with the same proof of Lemma 3.4 in \cite{be-de-la}, we get
   $$\frac{1}{\widehat{a_s}(h_s^{-1}(z_2(s,t))}=  s^2+O(s^3).$$
 The function $a_s$ we seek can be obtained by rescaling $\widehat{a_s}$ to ensure  
 $\int_0^{2\pi} c(t,s) dt = 1$ for all $s> 0$ small enough.
  \end{proof}

We are now in a position to apply the moment conditions \eqref{eqmomscal} to the system (\ref{eq:stat2}). We start with the function $G_s$:
\[\int_{\widetilde S_s} z^m \left(sz b_s(sz) \frac{\partial \rho}{\partial z_1}(w_s(sz),sz,\overline{w_s(sz)},s\overline z)\right) dz =0,\]
for all  $m\geq 0$, that is, using the form of the defining function $\rho$,
\[\int_{\widetilde S_s} z^j b_s(sz)\left(1 +\frac{i}{2}\eta(\imag w_s(sz),sz, s\overline z)- (\imag w_s(sz))\cdot \frac{\partial \eta}{\partial {z_1}}(\imag w_s(sz), sz, s\overline z)\right) dz =0 \]
for all  $j\geq 1$. We use the parametrization of $\widetilde S_s$ given by \eqref{eqpara}. With this parametrization, as observed in Remark \ref{rempara}, $s\widehat z_2$ is of order $O(s)$ and $\imag w_s(s\widehat z_2)=y_1$ of order $O(s^2)$. Since $\eta$ is of weighted order $O(6)$, the first term involving $\eta$ in the above integral is of order $O(s^6)$, while the second one is of order $O(s^7)$. Accordingly, we obtain, for $j\geq 1$
\begin{equation*}
\int_0^{T_s} r^j e^{i(j+1)t} b_s(sre^{it})(1+O(s^6))\left(\frac{\partial r}{\partial t} + ir\right)dt =0. 
\end{equation*}
Using Lemma \ref{lemr}, we have
\[\int_0^{T_s} e^{i(j+1)t} c(s, t) \left(1 + j k(t) s^{4} + O(s^{5})\right) \left(i +  \left(\frac{dk}{dt}(t)+ik(t)\right) s^4  + O(s^5)\right)dt =0.\]
Developing the product leads to 
\[\int_0^{T_s} e^{i(j+1)t} c(s, t) \left(1 - \frac{4}{3} j  a e^{2it} s^{4} + O(s^{5})\right) \left(i -4 i a e^{2it} s^4  + O(s^5)\right)dt= \]
\[=i\int_0^{T_s} e^{i(j+1)t} c(s, t) \left(1 - \frac{4}{3}(j+3)  a e^{2it} s^{4} + O(s^{5})\right) dt= 0\]
In order to apply Fourier Analysis, we apply the change of variables $t \mapsto \frac{2\pi}{T_s}t$ and, using Lemma \ref{lemper}, we obtain
\[\int_0^{2\pi} e^{i(j+1)\frac{T_s}{2\pi}t} c\left(s, \frac{T_s}{2\pi}t\right) \left(1 - \frac{4}{3}(j+3)  a e^{2i\frac{T_s}{\pi}t} s^{4} + O(s^{5})\right) dt\]
\[=i\int_0^{2\pi} e^{i(j+1)t} c(s, t) \left(1 - \frac{4}{3}(j+3)  a e^{2it} s^{4} + O(s^{5})\right) dt= 0.\]
 We may then expand $c(s,.)$ in its Fourier series 
$$c(s, t)=\sum_{k=-\infty}^{+\infty}\gamma_k(s)e^{ikt}$$ 
where $\gamma_{-k}=\overline \gamma_k$ for all $k\in \mathbb Z$ and, by Lemma \ref{lemc}, $\gamma_k$ is $\diffable{4}$ and satisfies $\gamma_0(s)\equiv 1$.
Inserting the Fourier expansion of $c(s, t)$ in 
$$\int_0^{2\pi} e^{i(j+1)t} c(s, t) \left(1 - \frac{4}{3}(j+3)  a e^{2it} s^{4} + O(s^{5})\right) dt= 0,$$
we deduce that 
\begin{equation}\label{eq:order4}
\overline \gamma_{j+1}(s)=O(s^4), 
\end{equation}
for all  $j\geq 1$. Taking the fourth derivative with respect to $s$, we get
\begin{equation*}
\sum_{\ell=0}^4\binom{4}{\ell}\left(\frac{d^\ell \overline \gamma_{j+1}}{ds^\ell}(s)\delta_{4}^\ell - \frac{4!}{\ell!}\left(\frac{4}{3}(j+3) a  \frac{d^\ell \overline \gamma_{j+3}}{ds^\ell}(s)\right) s^{\ell} + O(s^{\ell+1})\right)  =0,
\end{equation*}
where $\delta_{4}^\ell$ is the Kronecker symbol, which for $j=1$ leads to
\begin{equation*}
\frac{d^4 \overline \gamma_{2}}{ds^4}(s) -4!\frac{16}{3}  a \overline \gamma_{4}(s) = O(s),
\end{equation*}
implying that
\begin{equation}\label{eq:gamma2}
\frac{d^4 \overline \gamma_{2}}{ds^4}(s)=O(s).
\end{equation}

\vspace{0.5cm}

We now apply the moment conditions  \eqref{eqmomscal}  to the function $H_s$  in (\ref{eq:stat2}):
\[\int_{\widetilde S_s} z^m \left(sz b_s(sz) \frac{\partial \rho}{\partial z_2}(w_s(sz),sz,\overline{w_s(sz)},s\overline z)\right) dz =0,\]
for all  $m\geq 0$. Due to the form of the defining equation $\rho$, we get 
\[
  \begin{aligned} \int_{\widetilde S_s} z^j b_s(sz) \biggl(s\overline z &+ 2as^5z\overline z^4 + 
4 a s^5 z^3\overline z^2 \\ &+(\imag w_s(sz))\cdot \frac{\partial \eta}{\partial {z_2}}(\imag w_s(sz), sz, s\overline z)+\frac{\partial \delta}{\partial {z_2}}(sz,s\overline z)) \biggr) dz =0,\end{aligned}\]
for all $j\geq 1$. We once again parametrize $\widetilde S_s$ by \eqref{eqpara}. With this parametrization, the  term involving $\eta$ in the above integral is of order $O(s^7)$, and the one involving $\delta$ is of order $O(s^6)$.
We then obtain 
\[\begin{aligned}\int_0^{T_s}  e^{i(j+1)t} c(s, t)\biggl(r^{j}\overline r e^{-it}s&+r^{j}(2a r\overline r^4 e^{-3it}+4 a r^3\overline r^2 e^{it})s^5\\
                                                                                  &+r^jO(s^6)\biggr)\left(\frac{\partial r}{\partial t} + ir\right)dt=0.\end{aligned}\] 
Using once again Lemma \ref{lemr} and dividing by $is$ gives
\[\int_0^{T_s}  e^{i(j+1)t} c(s, t)\left( e^{-it} +s^4\left(\frac{2}{3}a  e^{-3it} +  \left(-\frac{4}{3}j  +4\right) a e^{it}\right) + O(s^5)\right) \cdot\]
\[\cdot\left(1 -4 a e^{2it}  s^4  + O(s^5)\right)dt=0,\]
and, after developing the product, and applying as above the change of variables $t \mapsto \frac{2\pi}{T_s}t$, we obtain
for  $j\geq 1$
\[\int_0^{2\pi} c(s, t)\left( e^{ijt} +\frac{2}{3}\left(a  e^{i(j-2)t} -\frac{4}{3}j a e^{i(j+2)t}\right)s^4 + O(s^5)\right)dt=0.\]
Once again, we integrate from $0$ to $2\pi$, insert the Fourier expansion of $c(s, t)$, and differentiate four times with respect to $s$, and obtain
\begin{equation*}
\sum_{\ell=0}^4\binom{4}{\ell}\left(\frac{d^\ell \overline \gamma_j}{ds^\ell}(s)\delta_{4}^\ell + \frac{4!}{\ell!}\left( \frac{2}{3}a \frac{d^\ell \overline \gamma_{j-2}}{ds^\ell}(s)  -\frac{4}{3}j    a \frac{d^\ell \overline \gamma_{j+2}}{ds^\ell}(s)\right) s^{\ell} + O(s^{\ell+1})\right) =0.
\end{equation*}
For $j=2$, this implies
\begin{equation*}
\frac{d^4 \overline \gamma_{2}}{ds^4}(s) - 4! \left(\frac{2}{3}a \overline \gamma_{0}(s) - 
\frac{8}{3}  a \overline \gamma_{4}(s)\right) = O(s).
\end{equation*}
Using (\ref{eq:order4}), (\ref{eq:gamma2}), we then deduce that $\displaystyle \frac{2}{3} a\overline\gamma_{0}(s)=\frac{2}{3} a=0$. This conclude the proof of Theorem \ref{thmmain}.

{\bf Funding information} 
Research of the first two authors was  supported by a Research Group Linkage Programme from the Humboldt Foundation, a URB grant from the American University of Beirut, and by the Center for Advanced Mathematical Sciences.
Research of the third author was supported by the Austrian Science Fund FWF, project AI4557-N.

\bibliographystyle{abbrv}
\bibliography{chains}

\end{document}